\documentclass[12pt, reqno]{amsart}

\usepackage{amsmath, amsfonts, amssymb}

\usepackage[latin1]{inputenc}
\usepackage{latexsym}

\usepackage{graphics}
\usepackage{color}

\newtheorem{theorem}{{\sc Theorem}}[section]

\newtheorem{lemma}[theorem]{{\sc Lemma}}

\theoremstyle{remark}
\newtheorem{remark}[theorem]{{\sc Remark}}

\newtheorem{example}[theorem]{{\sc Example}}
\newcommand{\E}{\mathbb E}
\newcommand{\V}{\mathbb V}

\newcommand{\R}{\mathbb R}
\newcommand{\N}{\mathbb N}

\begin{document}
\title[Moderate deviations via cumulants]{Moderate deviations via cumulants}
\bigskip
\author[Hanna D\"oring, Peter Eichelsbacher]{}

\maketitle
\thispagestyle{empty}
\vspace{0.2cm}

\centerline{\sc Hanna D\"oring\footnote{Technische Universit\"at Berlin, Institut f\"ur Mathematik, 
MA 767, D-10623 Berlin, Germany, {\tt hdoering@math.tu-berlin.de  }}, 
Peter Eichelsbacher\footnote{Ruhr-Universit\"at Bochum, Fakult\"at f\"ur Mathematik, 
NA 3/68, D-44780 Bochum, Germany, {\tt peter.eichelsbacher@ruhr-uni-bochum.de  } \\Both authors have been supported by Deutsche Forschungsgemeinschaft via SFB/TR 12. The first
author was supportet by the international research training group 1339 of the DFG.}}
\vspace{0.5cm}
\centerline{({\it Dedicated to the memory of Tomasz Schreiber})}

\vspace{2 cm}

\pagenumbering{roman}
\maketitle

\pagenumbering{arabic}
\pagestyle{headings}
\bigskip

{ {\bf Abstract:} } The purpose of the present paper is to establish moderate deviation principles
for a rather general class of random variables fulfilling certain bounds of the cumulants. We apply
a celebrated lemma of the theory of large deviations probabilities due to Rudzkis, Saulis and Statulevicius.
The examples of random objects we treat include dependency graphs, subgraph-counting statistics in Erd\H{o}s-R\'enyi random graphs
and $U$-statistics. Moreover, we prove moderate deviation principles for certain statistics appearing in random matrix theory,
namely characteristic polynomials of random unitary matrices as well as the number of particles in a growing box of random determinantal point processes
like the number of eigenvalues in the GUE or the number of points in Airy, Bessel, and $\sin$ random point fields.

\bigskip
\bigskip

\section{Introduction}

Since the late seventies estimations of cumulants have not only been studied
to show convergence in law, but have been studied to investigate
a more precise asymptotic analysis of the distribution via
the rate of convergence and large deviation probabilities,
see e.g. \cite{SaulisStratulyavichus:1989} and references therein.
In \cite{ERS:2009} it has been shown how to relate these bounds to prove a
moderate deviation principle for a class of counting functionals in models of geometric probability.
This paper provides a general approach to show moderate deviation principles 
via cumulants.

Let $X$ be a real-valued random variable with existing absolute moments. Then
$$
\left. \Gamma_j := \Gamma_j(X) :=(-i)^j \frac{d^j}{dt^j} \log \E\bigl[e^{i t X}\bigr] \right|_{t=0}
$$
exists for all $j\in\mathbb N$ and the term is called the {\it $j$th cumulant}
(also called semi-invariant) of $X$. Here and in the following $\E$ denotes the expectation 
of the corresponding random variable.
The method of moments results in a method of cumulants, saying that
if the distribution of $X$ is determined by its moments and $(X_i)_i$ are random variables with finite moments
such that $\Gamma_j(X_n) \to \Gamma_j(X)$ as $n \to \infty$ for every $j \geq 1$, then $(X_i)_i$ converges in distribution to $X$.
Hence if the first cumulant of $X_n$ converges to zero, the second cumulant to one
as well as all cumulants of $X_n$ bigger than $2$ vanish, then the sequence $(X_n)_n$  satisfies a Central Limit Theorem (CLT). 
Knowing additionally exact bounds of the cumulants one is able to describe
the asymptotic behaviour more precisely.  Let $Z_n$ be a real-valued random
variable with mean $\E Z_n=0$ and variance $\V Z_n=1$ and
\begin{equation} \label{cum1}
|\Gamma_j(Z_n)| \leq \frac{(j!)^{1+\gamma}}{\Delta^{j-2}}
\end{equation}
for all $j=3,4, \ldots$ and all $n \geq 1$ for fixed $\gamma\geq 0$ and $\Delta>0$. 
Here and in the following $\V$ denotes the variance of the corresponding random variable.
Denoting the standard normal distribution function by
$$\Phi(x):=\frac{1}{\sqrt{2\pi}} \int_{-\infty}^x e^{-\frac{y^2}{2}}dy\,,$$
one obtains the following bound for the Kolmogorov distance
$$
\sup_{x\in\mathbb R}\bigl|P(Z_n\leq x)-\Phi(x)\bigr|
\leq c_{\gamma} \, \Delta^{\frac{1}{1+2\gamma}}
$$
where $c_{\gamma}$ is a constant depending only on $\gamma$, see \cite[Lemma 2.1]{SaulisStratulyavichus:1989}.
By this result, the distribution function $F_n$ of $Z_n$ converges uniformly to $\Phi$ as $n \to \infty$.
Hence, when $x=O(1)$ we have 
\begin{equation} \label{mainratio}
\lim_{n \to \infty} \frac{1-F_n(x)}{1 -\Phi(x)} = 1.
\end{equation} 
One is interested to have - under additional conditions - 
such a relation in the case when $x$ depends on $n$ and tends to $\infty$ as $n \to \infty$. In particular,
one is interested in conditions for which the relation \eqref{mainratio} holds in the interval 
$0 \leq x \leq f(n)$, where $f(n)$ is a non-decreasing function such that $f(n) \to \infty$. 
If the relation hold in such an interval, we call the interval a zone of normal convergence. In the case of
partial sums of i.i.d. random variables with zero mean and finite positive variance, it can be shown
applying Mill's ratios that $f(n)$ can be chosen as $(1- \varepsilon) (\log n)^{1/2}$ for any $0 < \varepsilon <1$,
if the third absolute moment of $X_1$ is assumed to be finite (see \cite[Lemma 5.8]{Petrov:book}). 
Moreover, \eqref{mainratio} cannot be true in general since for the symmetric binomial distribution the numerator
vanishes for all $x > \sqrt{n}$. For i.i.d. partial sums the classical result due to Cram\'er is that if $\E e^{t |X_1|^{1/2}} < \infty$
for some $t >0$, \eqref{mainratio} holds with $f(n) = o(n^{1/6})$.
In \cite[Chapter 2]{SaulisStratulyavichus:1989}, relations of large deviations of the type \eqref{mainratio} are proved under the
condition \eqref{cum1} on cumulants with a zone of normal convergence of size proportional to $\Delta^{\frac{1}{1 + 2 \gamma}}$, see
Lemma 2.3 in \cite{SaulisStratulyavichus:1989}. 

The aim of this paper is to show that under the same type of condition on cumulants of random variables $Z_n$ moderate deviation
principles can be deduced. Actually we will go the detour via large deviation probabilities, showing that under condition
\eqref{cum1}, the deducible results on large deviations probabilities imply a moderate deviation principle. For partial sums $S_n$ of i.i.d. random variables
$(X_i)_i$ one can find in \cite{LiRosalsky:2004} the remark, that large deviation probability results imply asymptotic expansions for tail probabilities
$P(S_n \geq n \E(X_1) + n^{1/2} x)$ and $P(S_n \leq n \E(X_1) - n^{1/2} x)$ for $x \geq 0$ and $x = o(n^{1/2})$ and moreover, that these expansions imply
a moderate deviation principle. We have not found the general statement proven in the literature, that large deviation probability results imply
in general a moderate deviation principle. Our abstract result, Theorem \ref{thmcumulants}, is motivated by various applications. 
We will prove moderate deviation principles for a couple of statistics applying Theorem \ref{thmcumulants}. Some results will be improvements
of existing results, most of our examples are new moderate deviation results.

Let us recall the definition of a large deviation principle (LDP) due to Varadhan, see for example
\cite{DemboZeitouni:1998}. A sequence of probability measures $\{(\mu_n), n\in \mathbb N\}$ on a
topological space $\mathcal X$ equipped with a $\sigma$-field $\mathcal
B$ is said to satisfy the LDP with speed $s_n\nearrow \infty$ and
good rate function $I(\cdot)$ if the level sets $\{x: I(x)\leq
\alpha\}$ are compact for all $\alpha\in[0,\infty)$ and for all
$\Gamma\in\mathcal B$ the lower bound
$$
\liminf_{n\to\infty}  \frac{1}{s_n} \log \mu_n(\Gamma)
\geq - \inf_{x\in \operatorname{int}(\Gamma)} I(x)
$$
and the upper bound
$$
\limsup_{n\to\infty}  \frac{1}{s_n} \log \mu_n(\Gamma)
\leq - \inf_{x\in \operatorname{cl}(\Gamma)} I(x)
$$
hold. Here $\operatorname{int}(\Gamma)$ and
$\operatorname{cl}(\Gamma)$ denote the interior and closure of
$\Gamma$ respectively. We say a sequence of random variables satisfies
the LDP when the sequence of measures induced by these variables
satisfies the LDP. Formally a moderate deviation principle is nothing
else but the LDP. However, we will speak about a moderate deviation 
principle (MDP) for a sequence of random variables, whenever the scaling
of the corresponding random variables is between that of an ordinary Law
of Large Numbers and that of a Central Limit Theorem.

The following main theorem of this paper generalizes the idea in
\cite{ERS:2009} to use the method of cumulants to investigate moderate deviation principles:

\begin{theorem}\label{thmcumulants}
For any $n \in {\Bbb N}$, let $Z_n$ be a centered random variable with variance one and existing
absolute moments, which satisfies
\begin{equation}\label{eqcumulants}
\bigl| \Gamma_j(Z_n) \bigr| \leq (j!)^{1+\gamma} / \Delta_n^{j-2}
\quad\text{for all } j=3,4, \dots
\end{equation}
for fixed $\gamma\geq 0$ and $\Delta_n>0$.
Let the sequence $(a_n)_{n \geq 1}$ of real numbers grow to infinity, but slow enough such that
$$
\frac{a_n}{\Delta_n^{1/(1+2\gamma)}} \stackrel{n\to\infty}{\longrightarrow} 0
$$
holds.
Then the moderate deviation principle for
$\bigl(\frac{1}{a_n} Z_n\bigr)_n$
with speed $a_n^2$ and rate function $I(x)=\frac{x^2}{2}$
holds true. 
\end{theorem}

\noindent
The Theorem opens up the possibility to prove moderate deviations for a wide range
of dependent random variables. Before we will proceed, we will consider a moderate deviation principle 
for partial sums of independent, non-identically distributed random variables. Interesting enough, we have not
find any reference for the following result. 

\begin{theorem} \label{thmnotidentical}
Let $(X_i)_{i \geq 1}$ be a sequence of independent real-valued random variables with expectation zero and variances
$\sigma_i^2>0$, $i \geq 1$, and let us assume that $\gamma\geq 0$ and $K>0$ exist such that for all $i \geq 1$
\begin{equation}\label{momentenbedingungen}
\bigl| \E X_i^j\bigr| \leq (j!)^{1+\gamma} K^{j-2} \sigma_i^2
\quad\text{for all } j=3,4, \dots\, .
\end{equation}
Let $Z_n:=\frac{1}{\sqrt{\sum_{i=1}^n \sigma_i^2}}\sum_{i=1}^n X_i$. Then 
$\bigl(\frac{1}{a_n} Z_n \bigr)_{n \geq 1}$ satisfies the moderate
deviation principle with speed $a_n^2$ and rate function $\frac{x^2}{2}$ for any 
$1\ll a_n\ll \left(\frac{\sqrt{\sum_{i=1}^n \sigma_i^2}}{\displaystyle{2 \max\bigl\{K; \max_{1\leq i\leq n}}\{\sigma_i\}\bigr\}}\right)^{1/(1+2\gamma)}$.
\end{theorem}

Remark that condition \eqref{momentenbedingungen} is a generalization of the classical Bernstein condition ($\gamma =0$).
\begin{proof}
Using a relation between moments and cumulants, condition \eqref{momentenbedingungen} implies that the $j$-th cumulant of $X_i$ 
can be bounded by $(j!)^{1 + \gamma} (2 \max \{K, \sigma_i \})^{j-2} \sigma_i^2$. Hence it follows from the independence of the random variables $X_i$, $i \geq 1$,
that the $j$-th cumulant of $Z_n$ has the bound
\begin{equation}\label{cumulantenbedinungen}
|\Gamma_j(Z_n)| \leq (j!)^{1+\gamma} \left(
\frac{2 \max\bigl\{K; \max_{1\leq i\leq n}\{\sigma_i\}\bigr\}}{\sqrt{\sum_{i=1}^n \sigma_i^2}}\right)^{j-2} 
\, ,
\end{equation}
for details see for example \cite[Theorem 3.1]{SaulisStratulyavichus:1989}.
Thus for $Z_n$ the condition of Theorem \ref{thmcumulants} holds with
$$
\Delta_n=\frac{\sqrt{\sum_{i=1}^n \sigma_i^2}}{\displaystyle{2 \max\bigl\{K; \max_{1\leq i\leq n}\{\sigma_i\}\bigr\}}}\,.
$$
The result follows from Theorem \ref{thmcumulants}.
\end{proof}

\begin{remark}
If Cram\'er's condition holds, that is there
exists $\lambda>0$ such that $\E e^{\lambda |X_i|} < \infty$
holds for all $i\in\mathbb N$, then $X_i$ satisfies Bernstein's condition, which is the bound \eqref{momentenbedingungen} 
with $\gamma=0$, see for example \cite[Remark 3.6.1]{Yurinsky:1995}.
This implies \eqref{cumulantenbedinungen} and we can apply Theorem \ref{thmcumulants}
as above. Therefore Theorem \ref{thmcumulants} requires less restrictions on the random sequence
than Cram\'er's condition.
\end{remark}

The paper is organized as follows. Section 2 is devoted to applications for so called {\it dependency graphs}
including counting-statistics of subgraphs in Erd\H{o}s-R\'enyi random graphs. Section 3 presents applications to $U$-statistics.
Theorem \ref{thmcumulants} and Theorem \ref{thmnotidentical} will be applied in random matrix theory in Section 4. We will be able to reprove
moderate deviations for the characteristic polynomials for the COE, CUE and CSE matrix ensembles. Moreover
we will prove moderate deviations for determinantal point processes with applications in random matrix theory.
Finally, in Section 5 we present the proof of Theorem \ref{thmcumulants}.

\section{Applications to dependency graphs}\label{Application to dependency graphs}

Let ${\{X_{\alpha}\}}_{\alpha\in \mathcal I}$ be a family of random variables defined
on a common probability space. A {\it dependency graph} for ${\{X_{\alpha}\}}_{\alpha\in \mathcal I}$
is any graph $L$ with vertex set $\mathcal I$ which satisfies the following condition:
For any two disjoint subsets of vertices $V_1$ and $V_2$ such that there is no edge from any vertex in
$V_1$ to any vertex in  $V_2$, the corresponding collections of random
variables $\{X_{\alpha}\}_{\alpha\in V_1}$ and $\{X_{\alpha}\}_{\alpha\in V_2}$ are independent.

Let the {\it maximal degree} of a dependency graph $L$ be the maximum of the
number of edges coinciding at one vertex of $L$. The idea behind the usefulness of dependency graphs is that if the maximal degree
is not too large, one expects a Central Limit Theorem for the partial sums of the family  ${\{X_{\alpha}\}}_{\alpha\in \mathcal I}$.
We will consider moderate deviations. Note that there does not exist a unique dependency graph, for example the complete graph works for any set of random variables.

\begin{example} \label{stexample}
A standard situation is, that there is an underlying family of independent random variables $\{Y_i\}_{i \in \mathcal A}$, and each $X_{\alpha}$
is a function of the variables $\{Y_i\}_{i \in \mathcal A_{\alpha}}$, for some $\mathcal A_{\alpha} \subset \mathcal A$. With $\mathcal S = \{ \mathcal A_{\alpha}: 
\alpha \in \mathcal I \}$ the graph $L=L(\mathcal S)$ with vertex set $\mathcal I$ and edge set $\{ \alpha \beta: A_{\alpha} \cap A_{\beta} \not= \emptyset \}$
is a dependency graph for the family ${\{X_{\alpha}\}}_{\alpha\in \mathcal I}$. As a special case of this example, we will consider subgraphs
of an Erd\H{o}s-R\'enyi random graph. 
\end{example}

Another context, outside the scope of the present paper, in which dependency graphs are used is the Lov\'asz Local Lemma,
see \cite{AlonSpencer08}. Central limit theorems for $Z:=\sum_{\alpha \in \mathcal I} X_{\alpha}$ are obtained in \cite{BaldiRinott:1989}, see \cite[Theorem 9.6]{Steinbuch2010}
for corresponding Berry-Esseen bounds. We obtain the following bounds on cumulants of $Z$:

\begin{theorem}\label{lemmacumulantsrg}
Suppose that $L$ is a dependency graph for the family $\{X_{\alpha}\}_{\alpha \in \mathcal I}$ and that
$M$ is the maximal degree of $L$. Suppose further that $|X_{\alpha}| \leq A$ almost surely
for any $\alpha \in \mathcal I$ and some constant $A$. Let $\sigma^2$ be the variance of $Z:=\sum_{\alpha \in \mathcal I} X_{\alpha}$.
Then the cumulants $\Gamma_j$ of $\frac{Z}{\sigma}$ are bounded by
\begin{equation}\label{eqcumulantRG}
\bigl|\Gamma_j\bigr|\leq
(j!)^3 |\mathcal I| \, (M+1)^{j-1} \, (2eA)^j \frac{1}{\sigma^j}
\end{equation}
for all $j\geq 1$.
\end{theorem}

\begin{proof}
For notational reasons we consider without loss of generality the case where the index set $\mathcal I$ is chosen to be
$\mathcal I=\{1,\dots,N\}$ for any fixed natural number $N\in\N$. In \cite[Lemma 4]{Janson:1988} bounds for
the cumulants were given. Our main task is to obtain a bound, which gives the dependency of $j$ (and $j!$) as exact as possible. The first steps of our proof
can exactly be found in \cite[Lemma 4]{Janson:1988}.
Assuming the existence of the $m$-th moments of $X_1,\dots,X_j$ define the multi-linear function
$$
\kappa(X_1,\dots,X_j) := (-i)^j \frac{\partial^j}{\partial t_1 \cdots \partial t_j} \log \E \bigl[\exp(i t_1 X_1)\cdots \exp(i t_j X_j)\bigr] \Big|_{(t_1,\dots,t_j)=(0,\dots,0)} \,.
$$
Per definition for any random variable $X$ the cumulant is given by $\Gamma_j(X)=\kappa(\underbrace{X,\dots,X}_{j \text{ times}})$
and for the cumulant of $\frac{Z}{\sigma}$ we have
\begin{equation}
\Gamma_j
= \kappa\Bigl(\underbrace{\sum_{i=1}^{N}X_i,\dots,\sum_{i=1}^{N}X_i}_{j \text{ times}}\Bigr) \frac{1}{\sigma^j}
= \sum_{{i_1}=1}^{N}\dots  \sum_{{i_j}=1}^{N} \kappa(X_{i_1},\dots, X_{i_j}) \frac{1}{\sigma^j}\,.
\label{eqGesamtausdruckKumulante}
\end{equation}
Suppose that $X_{1},\dots,X_{m}$ are independent of $X_{m+1},\dots,X_{j}$
for any $1\leq m<j$, then
\begin{eqnarray*}
\kappa(X_{1},\dots, X_{j})
&=&
 (-i)^j \frac{\partial^j}{\partial t_1 \cdots \partial t_j} \log \E \bigl[\exp(i t_{1} X_{1})\cdots \exp(i t_{j} X_{j})\bigr] \Big|_{(0,\dots,0)}
\\
&=&
 (-i)^j \frac{\partial^j}{\partial t_{1} \cdots \partial t_{j}}
\log \E \bigl[\exp(i t_1 X_1)\cdots \exp(i t_m X_m)\bigr]\Big|_{(0,\dots,0)}
\\
&&{}+ (-i)^j \frac{\partial^j}{\partial t_{1} \cdots \partial t_{j}}
 \log \E \bigl[\exp(i t_{m+1} X_{m+1})\cdots \exp(i t_j X_j)\bigr] 
 \Big|_{(0,\dots,0)} 
\\&=& 0\,.
\end{eqnarray*}
Thus in \eqref{eqGesamtausdruckKumulante} we only have to consider those terms $\kappa(X_{i_1},\dots,X_{i_j})$ for which the corresponding $j$ vertices of $L$ (not necessarily 
distinct) form a connected subgraph.

For $1\leq q\leq j$, let $\sum_{I_1,\dots,I_q}$ denote the summation over all partitions
$I_m$ of $\{1,\dots,j\}$ into $m$ nonempty subsets, $1\leq m\leq q$.
The representation of a cumulant in \cite[Eq. (1.57)]{SaulisStratulyavichus:1989},
which was derived by Leonov and Shiryaev in 1959 via Taylor's expansion, 
gives
\begin{eqnarray}
\kappa(X_{i_1},\dots, X_{i_j})
&=& \sum_{q=1}^j \sum_{I_1,\dots,I_q}
  \left| (-1)^{q-1} (q-1)! \prod_{m=1}^q \E\left[\prod_{r\in I_m} X_{i_r}\right] \right|
  \nonumber\\
&\leq& \sum_{q=1}^j \sum_{I_1,\dots,I_q}
  \left| (-1)^{q-1} (q-1)! \prod_{m=1}^q \prod_{r\in I_m} \|X_{i_r}\|_{m_i}\bigr] \right|
\label{eqcumulantshoelder}
\end{eqnarray}
applying H\"older's inequality with $\sum_{i\in I_m}\frac{1}{m_i}=1$
and symbolizing $\bigl( \E |X_{i_r}|^{m_i}\bigr)^{1/m_i}$ by $\|X_{i_r}\|_{m_i}$.
Choosing $m_i=|I_m|$ and using the fact that $\|X_{i_r}\|_{m_i}\leq \|X_{i_r}\|_j$
for $m_i\leq j$ implies 
\begin{eqnarray}
\kappa(X_{i_1},\dots, X_{i_j})
&\leq&
\sum_{q=1}^j \sum_{I_1,\dots,I_q} \left| (-1)^{q-1} (q-1)! \prod_{m=1}^q \prod_{r\in I_m} \|X_{i_r}\|_j \right|
\nonumber
\\
&=&
\|X_{i_1}\|_j \cdots \|X_{i_j}\|_j \sum_{q=1}^j \sum_{I_1,\dots,I_q} \left| (-1)^{q-1} (q-1)! \right|
\,.\label{eqforeachcumulant}
\end{eqnarray}
The number of partitions of an set containing $j$ elements into $q$ parts is the Stirling number
$$\frac{1}{q!} \sum_{m=0}^q (-1)^{q-m} \left(q\atop m\right) m^j\,.$$
And inequality \eqref{eqforeachcumulant} implies
$$
\kappa(X_{i_1},\dots, X_{i_j})
\leq
\|X_{i_1}\|_j \cdots \|X_{i_j}\|_j \sum_{q=1}^j \frac{(q-1)!}{q!} \sum_{m=0}^q  \left(q\atop m\right) m^j \,.
$$
Since
$$\frac{1}{q} \sum_{m=0}^q  \left(q\atop m\right) m^j
\leq  \sum_{m=1}^q  \left(q\atop {\lfloor q/2\rfloor}\right) q^{j-1} 
= \left(q\atop {\lfloor q/2\rfloor}\right) q^j
\leq j^j \left(j\atop {\lfloor j/2\rfloor}\right)
$$
holds, we can apply $\left(j\atop {\lfloor j/2\rfloor}\right)
= \frac{j!}{ ({\lfloor j/2\rfloor})! ({\lceil j/2\rceil})!}
\leq \frac{j!}{ ({\lfloor j/2\rfloor})!^2}$
and the Stirling approximation $m! > \sqrt{2 \pi m} \left(\frac{m}{e}\right)^m$ to get
\begin{eqnarray}
\kappa(X_{i_1},\dots, X_{i_j})
&\leq&
\|X_1\|_j \cdots \|X_j\|_j \cdot j^{j+1} \frac{j!}{2\pi \frac{j}{2} \left(\frac{j}{2e}\right)^{2j/2}}
\nonumber\\
&\leq&
\|X_1\|_j \cdots \|X_j\|_j \cdot j! \cdot (2e)^j 
\quad \leq \quad j! (2e A)^j \,.
\label{eqcumulantsfastfertig}
\end{eqnarray}

Now we need to know the number of possible sets of $j$ vertices forming
a connected subgraph of $L$. If $v_1,\dots,v_j$ are $j$ such vertices,
then we can rearrange the indices such that each set $\{v_1,v_2\}$,
$\{v_1,v_2,v_3\}, \dots, \{v_1,\dots,v_j\}$ forms itself a connected subgraph
of $L$. There are at most $j!$ tuples of $j$ vertices associated to the
same ordering. 
There are $N$ ways of choosing $v_1$. The vertex $v_2$ must equal $v_1$ or
be connected to $v_1$, for which we have the choice of at most $M$ possible
vertices.
Similarly, $v_3$ either equals $v_1$ or $v_2$ or is connected to one of them.
For this choice we have at most $2+2 M=2(M+1)$ possibilities.
Continuing this way we see that there are at most
$$
j! N (M+1) 2(M+1) \cdots (j-1)(M+1)
=j! (j-1)! N (M+1)^{j-1}
$$
choices of $j$ vertices forming a connected subgraph in $L$.

Inserting this estimation and the bound in \eqref{eqcumulantsfastfertig}
into equation \eqref{eqGesamtausdruckKumulante} completes the proof of Theorem
\ref{lemmacumulantsrg}.
\end{proof}

\subsection{Subgraphs in Erd\H{o}s-R\'enyi random graphs}

Consider an Erd\H{o}s-R{\'e}nyi random graph with $n$ vertices, where
for all $\left(n \atop 2\right)$ different pairs of vertices the
existence of an edge is decided by an independent Bernoulli experiment
with probability $p$. For each $i\in\{1,\dots,\left({{n}\atop{2}}\right)\}$,
let $X_i$ be the random variable determining if the edge $e_i$ is present,
i.e. $P(X_i=1)=1-P(X_i=0)=p(n) =:p$. The model is called ${\Bbb G}(n,p)$.
The following statistic counts the number of subgraphs isomorphic to a fixed
graph $G$ with $k$ edges and $l$ vertices
\begin{equation} \label{wdef}
W=
\sum_{1\leq \kappa_1<\dots < \kappa_k \leq \left({{n}\atop{2}}\right)}
  1_{\{(e_{\kappa_1},\dots,e_{\kappa_k})\sim G\}} \left(\prod_{i=1}^k X_{\kappa_i} \right)
\:.
\end{equation}
Here $(e_{\kappa_1}, \ldots, e_{\kappa_k})$ denotes the graph with edges
$e_{\kappa_1}, \ldots, e_{\kappa_k}$ present and $A \sim G$ denotes the fact that the subgraph $A$
of the complete graph is isomorphic to $G$. Here and in the following we speak about connected subgraphs only.
Let the constant $a := \rm{aut}(G)$ denote the order of the automorphism group of $G$.
The number of copies of $G$ in $K_n$, the complete graph with $n$
vertices and $\left(n \atop 2\right)$ edges, is given by
$\left(n \atop l\right) l!/a$ and the expectation of $W$ is equal to
$\E[W] = \frac{\left(n \atop l\right) l!}{a} p^k = {\mathcal O}(n^l p^k) \:$.
It is easy to see that $P(W> 0)= o(1)$ if $p\ll n^{-l/k}$.
Moreover, for the graph property that $G$ is a subgraph, the probability that
a random graph possesses it jumps from $0$ to $1$ at the threshold probability
$n^{-1/m(G)}$, where $m(G)= \max \left\{ \frac{e_H}{v_H} : H\subseteq G, v_H >0 \right\}$,
$e_H, v_H$ denote the number of edges and vertices of $H\subseteq G$,
respectively, see \cite{JLR:2000}.
Ruci{\'n}ski proved in \cite{Rucinski:1988} that
$\frac{W-\E(W)}{\sqrt{\V(W)}}$
converges in distribution to a standard normal distribution if and only if
\begin{equation}
n p^{m(G)}\stackrel{n\to\infty}{\longrightarrow}  \infty
\quad\text{and}\quad
n^2 (1-p) \stackrel{n\to\infty}{\longrightarrow}  \infty\:.
\end{equation}
An upper bound for {\it lower tails} was proven by Janson \cite{Janson:1990}, applying the FKG-inequality. A comparison of {\it seven} different
techniques proving bounds for {\it the infamous} upper tail can be found in \cite{JansonRucinski:2002}, see also \cite{Chatterjee:2010} 
for a recent improvement. The large deviation principle for subgraph count statistics in Erd\H{o}s-R\'enyi random graphs with
fixed $p$ are solved in \cite{ChatterjeeVaradhan:2010}.

As a special case of Example \ref{stexample}, let $\{H_{\alpha}\}_{\alpha \in \mathcal I}$ be given subgraphs of the complete graph $K_n$ and let $I_{\alpha}$
be the indicator that $H_{\alpha}$ appears as a subgraph in ${\Bbb G}(n,p)$, that is, $I_{\alpha} = 1_{\{ H_{\alpha} \subset {\Bbb G}(n,p) \}}$, $\alpha \in \mathcal I$.
Then $L(S)$ with $S = \{e_{H_{\alpha}} : \alpha \in \mathcal I\}$ is a dependency graph with edge set $\{ \alpha \, \beta: e_{H_{\alpha}} \cap e_{H_{\beta}} \not= \emptyset \}$.
Here we take the family of subgraphs of $K_n$ that are isomorphic to a fixed graph $G$, denoting by $\{G_{\alpha}\}_{\alpha \in A_n}$. 
Consider $X_{\alpha} = I_{\alpha} - {\Bbb E} I_{\alpha}$ and define the graph $L_n$ by connecting every pair of indices $\alpha$ and $\beta$ such that the corresponding
graphs $G_{\alpha}$ and $G_{\beta}$ have a common edge. This is evidently a dependency graph for $(X_{\alpha})_{\alpha \in A_n}$; see \cite[Example 6.19]{JLR:2000}.
Note that the subgraph count statistic $W-\E W$ given in \eqref{wdef} is equal to the sum of all $X_{\alpha}$, $1\leq \alpha\leq A_n$.

We will be able to prove the following moderate deviation principle for the subgraph count statistic:

\begin{theorem}\label{thmMDPtriangle}
Let $G$ be a fixed graph with $k$ edges and $l$ vertices.
Let $(a_n)_n$ be a sequence with
$$
1 \ll a_n \ll
\Bigl(\frac{n \bigl(p^{k-1} \sqrt{p (1-p)}\bigr)^3}{8 k^2 e^3}\Bigr)^{1/5}
\,,
$$
where $e$ is Euler's number.
Then the scaled subgraph count statistics
$\bigl(\frac{1}{a_n} \frac{W-\E W}{\sqrt{\V W}}\bigr)_n$
satisfy the moderate deviation principle with speed $a_n^2$ and rate
function $x^2/2$ if
\begin{equation}\label{eqprobbedingung}
n^2 p^{3(2k-1)} (1-p)^3\stackrel{n\to \infty}{\longrightarrow} \infty
\end{equation}
holds.
\end{theorem}

\begin{remark}
Condition \eqref{eqprobbedingung} on $p(n)$ assures that $(a_n)_n$ grows to infinity.
Moderate deviations for the subgraph count statistic of Erd{\H o}s-R{\'e}nyi
random graphs are already considered in \cite{DoeringEichelsbacher:2009}
studying the log-Laplace transform via martingale differences and using the
G\"artner-Ellis Theorem. The stated moderate deviation principle in Theorem \ref{thmMDPtriangle} is on one
hand valid for more probabilities $p(n)$ than in \cite[Theorem 1.1]{DoeringEichelsbacher:2009}.
But on the other hand the scaling $\beta_n :=a_n \sqrt{\V W}$ has a smaller range in comparison to \cite{DoeringEichelsbacher:2009}:
Using $\text{const.} \, n^{2l-2} p^{2k-1} (1-p)
\leq \V W \leq \text{const.} \, n^{2l-2} p^{2k-1} (1-p)$
\label{Rucinskivar}
(see \cite[2nd section, page 5]{Rucinski:1988}) 
the scaling in Theorem \ref{thmMDPtriangle} is equal to
$$
n^{l-1} p^{k-1} \sqrt{ p(1-p) } \ll
a_n \sqrt{\V W}
\ll n^{l-\frac{4}{5}} \bigl(p^{k-1} \sqrt{ p(1-p) }\bigr)^{8/5}\,.
$$
Where the scaling in Theorem \cite[Theorem 1.1]{DoeringEichelsbacher:2009} is bounded by:
$$
n^{l-1} p^{k-1} \sqrt{ p(1-p) } \ll \beta_n \ll n^{l} \left( p^{k-1} \sqrt{p(1-p)} \right)^4\,.
$$
\end{remark}

\begin{proof}[Proof of Theorem~\ref{thmMDPtriangle}]
In order to prove Theorem \ref{thmMDPtriangle} we apply
Theorem \ref{lemmacumulantsrg} to show that the conditions
of Theorem \ref{thmcumulants} are satisfied.
Let us consider the subgraph count statistic in an
Erd{\H o}s-R{\'e}nyi random graph for any fixed subgraph $G$ with $l$ vertices
and $k$ edges and its associated dependency graph $L_n$ defined as above.
Let $M_n$ be the maximal degree of the dependency graph $L_n$. Thus to
determine $M_n$ we need to bound the maximal number of subgraphs isomorphic
to $G$ having at least one edge in common with a fixed subgraph $G'$ which
is itself isomorphic to $G$. For every subgraph $G'$, isomorphic to $G$, we
have to consider one of the $k$ edges of $G'$ to be the common edge.
Accordingly we can choose $l-2$ further vertices out of $n-2$ possible vertices
-- which justifies a factor $(n-2)_{l-2}:=(n-2) (n-1) \cdots (n-l-1)$.
We can substract one solution, because we do not count $G'$ itself and achieve
$$
M_n \leq k (n-2)_{l-2} -1 \leq k n^{l-2} -1\,.
$$
The number $N_n$ of the subgraphs in $K_n$ which are isomorphic to $G$ satisfies the inequality
$$
\frac{n (n-1)\cdots (n-l-1)}{l (l-1)\cdots 1}=\left({n}\atop{l}\right)
\leq N_n \leq n_l=n (n-1)\cdots (n-l-1)\,.
$$
As stated on page \pageref{Rucinskivar} the variance $\sigma_n^2=\V W$ of
$\sum_{\alpha=1}^{N_n} Y_{\alpha}= W-\E W$ is bounded by a constant times
$n^{2l-2} p^{2k-1} (1-p)$.
For the cumulants of $\frac{W-\E W}{\sqrt{\V W}}$ it follows with \eqref{eqcumulantRG} that, for $j\geq 3$,
\begin{eqnarray}
\bigl|\Gamma_j\bigr|
&\leq&
(j!)^3 n^l \bigl(k n^{l-2}\bigr)^{j-1} (2e)^j \frac{1}{\bigl(\text{const.} n^{l-1} p^{k-1} \sqrt{p (1-p)}\bigr)^j}\nonumber\\
&=&
(j!)^3 \frac{1}{n^{j-2}} k^{j-1} (2e)^j \frac{1}{\bigl(\text{const.} p^{k-1} \sqrt{p (1-p)}\bigr)^j}\nonumber\\
&\leq&
(j!)^3 \left(\frac{8 k^2 e^3}{n \bigl(\text{const.} p^{k-1} \sqrt{p (1-p)}\bigr)^3}\right)^{j-2}\,.
\label{subgraphcumulant}
\end{eqnarray}
In the last inequality we used the fact that $3(j-2)\geq j$ is equivalent to
$j\geq 3$.
This implies
that condition \eqref{eqcumulants} is satisfied for $\gamma=2$ and
$$\Delta_n
=\frac{n \bigl(\text{const.} p^{k-1} \sqrt{p (1-p)}\bigr)^3}{8 k^2 e^3}\,.$$
$\Delta_n$ increases if $n^2 p^{3(2k-1)} (1-p)^3\stackrel{n\to \infty}{\longrightarrow} \infty$.
So in the following we can only consider this case.
Now we choose a sequence $(a_n)_n$ such that
$$
1 \ll a_n \ll \Delta_n^{1/(1+2\gamma)}
= \left(\frac{n \bigl(p^{k-1} \sqrt{p (1-p)}\bigr)^3}{8 k^2 e^3}\right)^{1/5}\,,
$$
and apply Theorem \ref{thmcumulants}, which ends the proof of Theorem \ref{thmMDPtriangle}.
\end{proof}

\begin{remark}\label{CLTcumulants}
As mentioned in the introduction, the cumulant bounds \eqref{eqcumulants}
imply a Central Limit Theorem if $\lim_{n\to\infty}\Delta_n=\infty$.
Moreover applying \cite[Lemma 2.1]{SaulisStratulyavichus:1989} and inequality
\eqref{subgraphcumulant} proves the following bound for the Kolmogorov distance:
\begin{equation*}
\sup_{x\in\mathbb R} \Bigl| P\Bigl(\frac{W-\E W}{\sqrt{\V W}}\leq x\Bigr)-\Phi(x)\Bigr|
\leq 108 \left(\frac{\sqrt{2}}{6} \Delta_n \right)^{-\frac{1}{1+2\gamma}}
\leq \frac{\text{const.}}{n^{1/5}  \bigl(p^{k-1} \sqrt{p (1-p)}\bigr)^{3/5}}\,.
\end{equation*}
This bound is weaker than the inequality induced in \cite{BarbourKaronskiRucinski:1989} via Stein's method.
For some improvements see \cite{Goldstein}.
\end{remark}

\subsection{Another example of a dependency graph}

Let $X_i$, $i \geq 1$, be independent centered random variables
with existing variances $\V X_i \geq \varepsilon$ for any $\varepsilon>0$ and define $Z_n:=\sum_{i=1}^n X_i X_{i+1}$.
Let $A$ be a constant such that $|X_i|\leq \sqrt{A}$ almost surely. Let $(a_n)_n$
be a divergent sequence where $a_n\ll n^{1/10}$. Then $\Bigl(\frac{1}{a_n \sqrt{\V Z_n}} Z_n\Bigr)_{n\in\mathbb N}$
satisfies the moderate deviation principle with speed $a_n^2$ and rate function $I(x)=x^2/2$.

\begin{proof}
Set $Y_i:= X_i X_{i+1}$ for all $i=1,\dots,n$. $Y_i$ is independent of $Y_j$ for all
$j$ not equal to $i-1$ and $i+1$. Let $L_n$ be the graph with vertex set
$\{1,\dots,n\}$ and edges between $1$ and $2$, $2$ and $3$, \dots as well as between
$n-1$ and $n$. $L_n$ is a dependency graph of $\{Y_i\}_{i=1}^{N}$ with $N=n$ and $M=2$.
The variance $\sigma_n^2=\V Z_n$ is bigger or equal than a constant times $n$:
\begin{eqnarray*}
\V Z_n &=& \sum_{i,j =1}^n \E \bigl[Y_i Y_j\bigr] = \sum_{i,j =1}^n \E[X_i X_{i+1} X_j X_{j+1}]\\
&=& \sum_{i=1}^n \E\bigl[X_i^2 X_{i+1}^2\bigr]
+ \sum_{i=2}^n \E\bigl[X_i X_{i+1} X_{i-1} X_i\bigr]
+ \sum_{i=1}^{n-1} \E\bigl[X_i X_{i+1} X_{i+1} X_{i+2}\bigr]\\
&=& \sum_{i=1}^n \sqrt{\V(X_i) \V(X_{i+1})} + 2 \sum_{i=2}^n \E[X_{i-1}]  \E[X_i^2] \E[X_{i+1}]\\
&=& \sum_{i=1}^n \sqrt{\V(X_i) \V(X_{i+1})}
\geq n \min_{i=1,\dots,n+1} \V(X_i)
\end{eqnarray*}
due to the independence of $X_1,\dots,X_{n+1}$ and the fact that their expectations
are equal to zero. In particular, for independent and identically distributed random
variables $X_i$, we have $\V Z_n= \text{const.} \cdot n$.
Using Theorem \ref{lemmacumulantsrg} we have a bound for the cumulant
of $\frac{1}{\sqrt{\V Z_n}} \sum_{i=1}^n Y_i =\frac{1}{\sqrt{\V Z_n}} Z_n$:
$$
\bigl|\Gamma_j\bigr| \leq (j!)^3 \left( \frac{A^3}{\text{const.} \sqrt{n}}\right)^{j-2}\,.
$$
Now we can apply Theorem \ref{thmcumulants} with $\gamma=2$, $\Delta_n= \text{const.} \sqrt{n}$
and a sequence $(a_n)_n$ satisfying
$\frac{a_n}{\sqrt{n}^{1/5}}\stackrel{n\to\infty}{\longrightarrow}0$.
This proves the claim.
\end{proof}

\section{Application to non-degenerate $U$-statistics}
\label{Application to non-degenerate U-statistics}
Let $X_1,\dots,X_n$ be independent and identically distributed random
variables with values in a measurable space $\mathcal X$. For a
measurable and symmetric function $h:{\mathcal X}^m\to \R$ we define
$$
U_n(h):= \frac{1}{\left(n\atop m\right)}
\sum_{1\leq i_1<\dots<i_m \leq n} h(X_{i_1},\dots,X_{i_m})\:,
$$
where symmetric means invariant under all permutations of its arguments.
$U_n(h)$ is called a {\it U-statistic} with {\it kernel} $h$ and
{\it degree} $m$.
Define the conditional expectation for $c=1,\dots,m$ by
\begin{eqnarray*}
h_c(x_1,\dots,x_c)
&:=&\E\bigl[ h(x_1,\dots,x_c,X_{c+1},\dots, X_m)\bigr]
\\
&=&\E\bigl[ h(X_1,\dots, X_m)\big| X_1=x_1,\dots, X_c=x_c\bigr]
\end{eqnarray*}
and the variances by $\sigma_c^2:=\V\bigl[h_c(X_1,\dots,X_c)\bigr]$.
A U-statistic is called {\it degenerate of order $d$} if and only if $0=\sigma_1^2 = \cdots = \sigma_d^2 < \sigma_{d+1}^2$ and
{\it non-degenerate} if $\sigma_1^2>0$. As is well known, the weak limits of appropriately scaled $U$-statistics depend
on the order of degeneracy. By the Hoeffding-decomposition (see for example \cite{Lee:1990}), we know
that for every symmetric function $h$, the $U$-statistic can be decomposed into a sum
of degenerate $U$-statistics of different orders. 
In the degenerate case the linear term of this decomposition
disappears. On the level of moderate deviations, in \cite{EichelsbacherSchmock:2003} the MDP for non-degenerate $U$-statistics
is investigated; the proof used the fact that the linear term 
in the Hoeffding-decomposition is leading in the non-degenerate case. 
Moreover in \cite{EichelsbacherSchmock:2003},
moderate deviation principles for Banach-space valued degenerate $U$-statistics were established, with bon-convex rate functions. 

In the present paper the observed U-statistics are assumed to be non-degenerate. The main result is:
 
\begin{theorem}{\it (Moderate deviations for non-degenerate $U$-statistics)}
\label{thm2.14}\newline
Let $X_1, X_2, \dots$ be a sequence of independent and identically distributed random variables and
$$
U_n(h)= \frac{1}{\left(n\atop 2\right)} \sum_{1\leq i_1 < i_2 \leq n} h(X_{i_1}, X_{i_2})
$$
a non-degenerate $U$-statistic of degree two. Let $\sigma_1^2:= \V \left( \E[h(X_1,X_2)| X_1] \right) <\infty$ and suppose that there exist constants $\gamma\geq1$ and $C>0$ such that
\begin{equation}\label{UStatKum}
\E\bigl[|h(X_1,X_2)|^j\bigr]\leq C^j (j!)^{\gamma}
\end{equation}
for all $j\geq 3$.
Defining
$$
C(\sigma_1) :=
\left\{\begin{array}{ll}
\frac{C}{\sigma_1}&\text{, if } C\leq \sigma_1\\
\frac{C^3}{\sigma_1^3}&\text{, if } C> \sigma_1\\
\end{array}
\right.
$$
and $\Delta_n:=\left(\frac{\sqrt{n}}{2 \sqrt{2} e C(\sigma_1)}\right)$,
let $(a_n)_n$ be a sequence growing to infinity such that
\begin{equation}\label{eqna_nDelta}
\frac{a_n}{\Delta_n^{1/(1+2\gamma)}} \stackrel{n\to\infty}{\longrightarrow} 0\, .
\end{equation}
Then $\Bigl(\frac{U_n}{a_n \sqrt{\V(U_n)}}\Bigr)_{n\in\mathbb N}$ satisfies the moderate deviation principle with speed $a_n^2$ and rate function
$I(x)=\frac{x^2}{2}$.
\end{theorem}

\begin{remark}\label{Remarkustat}
Let us discuss the conditions \eqref{UStatKum} and \eqref{eqna_nDelta} in detail.

\noindent
{\bf a.}
In \cite{EichelsbacherSchmock:2003} a moderate deviation
principle for degenerate and for non-degenerate U-statistics with a kernel
function $h$, which is bounded or satisfies exponential moment conditions,
was considered (see also \cite{Eichelsbacher:2001}).
In \cite{EichelsbacherSchmock:2003} the exponential moment conditions for a non-degenerate 
$U$-statistic of degree two reads as follows: the function $h_1$ of the leading term in the Hoeffding-decomposition
has to satisfy the weak Cram\'er condition: $\int \exp(\alpha \|h_1\|) dP < \infty$
for a $\alpha >0$. Moreover $h_2$ has to satisfy the condition that there exists at least one $\alpha_h >0$ such that
$\int \exp(\alpha_h \|h_2\|^2) dP^2 < \infty$. The MDP in  \cite{EichelsbacherSchmock:2003}
was proved for $1 \ll a_n \ll \sqrt{n}$. Since the leading term of the Hoeffding decomposition is a partial sum of
i.i.d. random variables, the weak Cram\'er condition on $h_1$ can be relaxed. A necessary and sufficient condition
is given in \cite{EichelsbacherLoewe:2003} which is
$$
\limsup_{n \to \infty} \frac{1}{a_n^2} \log \bigl( n P \bigl( |h_1(X_1)| > \sqrt{n} a_n \bigr) \bigr) = - \infty.
$$
The strong condition on $h_2$ is due to the fact, that a Bernstein-type inequality for the degenerate
part of the Hoeffding-decomposition was applied, see \cite[Theorem 3.26]{EichelsbacherSchmock:2003}.
Unfortunately is is not obvious how to compare condition \eqref{UStatKum} with the conditions in \cite{EichelsbacherSchmock:2003}.
Condition \eqref{UStatKum} is a Bernstein-type condition on the moments of $h$, which is equivalent to a weak Cram\'er condition
on $h$. We haven't no assumptions on $h_2$, hence \eqref{UStatKum} seems to be weaker. On the other side, even in the
case of the best bounds ($\gamma =1$) in \eqref{UStatKum}, our result is restricted to $1 \ll a_n \ll n^{1/6}$. The prize
of less restrictive conditions on $h$ seem to be that the moderate deviation principle holds in a smaller scaling-interval.
Our Theorem is an improvement of \cite{EichelsbacherSchmock:2003} for some $a_n$.

\noindent
{\bf b.}
We can also compare the result in Theorem \ref{thm2.14} with the result in
\cite[Theorem 3.1]{DoeringEichelsbacher:2009}, which was deduced via the Laplace transform. Let the kernel function $h$ be bounded.
Obviously condition \eqref{UStatKum} is fulfilled with  $\gamma=1$ and according to Theorem \ref{thm2.14} 
the object $\Bigl(\frac{U_n}{a_n \sqrt{\V(U_n)}}\Bigr)_n$ satisfies the MDP with speed $a_n^2$ and rate function
$I(x)=\frac{x^2}{2}$ for every sequence $(a_n)_n$ growing to infinity slow enough such that $1 \ll a_n\ll n^{1/6}$.

Let $(b_n)_n$ be a sequence satisfying $\sqrt{n}\ll b_n \ll n$. From \cite[Theorem 3.1]{DoeringEichelsbacher:2009} it follows that 
$\bigl(\frac{n}{b_n}U_n\bigr)_n$ satisfies the MDP with speed $\frac{b_n^2}{n}$ and rate function
$I(x)=\frac{x^2}{8\sigma_1^2}$. Choosing $b_n=n a_n \sqrt{\V U_n}$ in \cite[Theorem 3.1]{DoeringEichelsbacher:2009} requires
the scaling $\sqrt{n}\ll b_n =n a_n \sqrt{\V U_n} \ll n \,$.
Applying that $n\V U_n = 4\sigma_1^2 + O\bigl(\frac{1}{n}\bigr)$ gives
$1\ll a_n \ll \sqrt{n}$. From \cite[Theorem 3.1]{DoeringEichelsbacher:2009} we obtain, that
$\frac{n}{b_n}U_n=\frac{U_n}{a_n \sqrt{\V(U_n)}}$ satisfies the MDP with speed
$n a_n^2 \V U_n= a_n^2 4 \sigma_1^2 + O\bigl(\frac{a_n^2}{n}\bigr)$ 
and rate function $I(x)=\frac{x^2}{8\sigma_1^2}$. This is the same result
as stated above via Theorem \ref{thm2.14}. Therefore the MDP via the log-Laplace transform holds for a larger scaling range.
But \cite[Theorem 3.1]{DoeringEichelsbacher:2009} assumed bounded $U$-statistics, and thus Theorem \ref{thm2.14} is valid for more 
general kernel functions $h$ for some $a_n$.
\end{remark}

\begin{proof}
According to \cite{Alesk:1990}, see \cite[Lemma 5.3]{SaulisStratulyavichus:1989},
the cumulant of $U_n$ can be bounded by
$$
|\Gamma_j(U_n)| < 2 e^{2(j-2)}\frac{2^j-1}{j} C^j (j!)^{1+\gamma} \frac{1}{n^{j-1}}
$$
for all $j=1,2,\dots,n-1$ and $n\geq 7$. The quite involved proof is presented in \cite{SaulisStratulyavichus:1989}.
The variance for the non-degenerate $U$-statistic is given by
$\V (U_n)= \frac{4 \sigma_1^2}{n} \frac{n-2}{n-1} + \frac{2 \sigma_2^2}{n(n-1)}$,
see Theorem 3 in \cite[chapter 1.3]{Lee:1990}. Therefore it exists an $n_0\geq 7$
big enough such that $\sqrt{\V (U_n)}\geq \frac{e \sigma_1}{\sqrt{2n}}$.
The following bound holds for the cumulants of $\frac{U_n}{\sqrt{\V(U_n)}}$:
$$
|\Gamma_j|\leq
(j!)^{1+\gamma} \left(\frac{2 \sqrt{2} e C(\sigma_1)}{\sqrt{n}}\right)^{j-2}
$$
for all $j= 3,\dots, n-1$ and $n\geq n_0$.
Applying Theorem \ref{thmcumulants}, \eqref{eqna_nDelta}, with   $\Delta_n=\left(\frac{\sqrt{n}}{2 \sqrt{2} e C(\sigma_1)}\right)$ 
is a sufficient condition for the moderate deviation principle.
\end{proof}

\begin{remark}
Let us remark, that known precise estimates on cumulants will enable us to prove moderate deviation principles
for further probabilistic objects. Examples are polynomial forms, Pitman polynomial estimators and multiple stochastic integrals (see \cite{SaulisStratulyavichus:1989}).
This will be not the topic of this paper.
\end{remark}

\bigskip

\section{Moderate deviations for the characteristic polynomials in the circular ensembles}

In the last decade, a huge number of results in random matrix theory were proved. Some of the results were extrapolated
to make interesting conjectures on the behaviour of the Riemann zeta function on the critical line. It is known that random matrix
statistics describe the local statistics of the imaginary parts of the zeros high up on the critical line. The random matrix
statistic considered for this conjectural understanding of the zeta-function  is the characteristic polynomial $Z(\theta) := Z(U,\theta) =
\det\bigl(I-U e^{-i \theta}\bigr)$ of a unitary $n \times n$ matrix $U$. The matrix $U$ is considered as a random variable in the 
{\it circular unitary ensemble} (CUE), that is, the unitary group $U(n)$ equipped with the unique translation-invariant (Haar) probability measure.
In \cite{KeatingSnaith:2000} exact expressions for any matrix size $n$ are derived for the moments of $|Z|$ and from these the asymptotics of the value 
distribution and cumulants of the real and imaginary parts of $\log Z$ as $n \to \infty$ are obtained. In the limit, these distributions are independent and Gaussian.
In \cite{KeatingSnaith:2000} the results were generalized to the circular orthogonal (COE) and the circular symplectic (CSE) ensembles. The goal of this section is to
prove a moderate deviation principle for the appropriately rescaled $\log Z$ for the three classical circular ensembles applying Theorem \ref{thmcumulants}.
Remark that our result is known for CUE, see \cite[Theorem 3.5]{HughesKeatingOConnell:2001}, see Remark \ref{remarkchar}.
We present a different proof and generalize the result to the COE and CSE ensembles. We start with the representation of $Z(U, \theta)$ in terms
of the eigenvalues $e^{i \theta_k}$ of $U$:
$$
Z(U,\theta)= \det\bigl(I-U e^{-i \theta}\bigr) = \prod_{k=1}^n \bigl(1-e^{i(\theta_k-\theta)}\bigr).
$$
Let $Z$ now represent the characteristic polynomial of an $n \times n$ matrix $U$ in either the CUE ($\beta=2$), the COE ($\beta=1$), or the CSE ($\beta=4$).
The $C \beta E$ average can then be performed using the joint probability density for the eigenphases $\theta_k$
$$
\frac{(\beta/2)!^n}{(n \beta/2)! (2 \pi)^n} \prod_{1 \leq j < m \leq n} |e^{i \theta_j} - e^{i \theta_m}|^{\beta}.
$$
Hence the $s$-moment of $|Z|$ is of the form
$$
\langle |Z|^s \rangle_{\beta} = \frac{(\beta/2)!^n}{(n \beta/2)! (2 \pi)^n} \int_0^{2 \pi} \cdots \int_0^{2 \pi} d\theta_1 \cdots d\theta_n 
\prod_{1 \leq j < m \leq n} |e^{i \theta_j} - e^{i \theta_m}|^{\beta} \times \bigg| \prod_{k=1}^n \bigl(1 - e^{i(\theta_k - \theta)} \bigr) \bigg|^s.
$$
This integral can be evaluated using Selberg's formula, see \cite{Mehta:book}, which leads to
$$
\langle |Z|^s \rangle_{\beta} = \prod_{j=0}^n \frac{\Gamma(1 + j \beta /2) \Gamma(1 + s + j \beta /2)}{(\Gamma( 1 + s/2 + j \beta /2))^2}
$$
denoting the gamma function by $\Gamma$ (without an index).
Hence $\log \langle |Z|^s \rangle_{\beta}$ has an easy form and equals at the same time by definition $\sum_{j \geq 1} \frac{\Gamma_j(\beta)}{j!} s^j$,
where $\Gamma_j(\beta)= \Gamma_j(\Re \log Z)$ denotes the $j$-th cumulant of the distribution of the real part of $\log Z$ under $C \beta E$. Differentiating $\log \langle |Z|^s \rangle_{\beta}$
one obtains
\begin{equation} \label{vorbereitung}
\Gamma_j (\beta) = \frac{2^{j-1} -1}{2^{j-1}} \sum_{k=0}^{n-1} \psi^{(j-1)}(1 + k \beta /2),
\end{equation}
where 
$$
\psi^{(j)}(z):= \frac{d^{j+1} \log \Gamma(z)}{dz^{j+1}}
= (-1)^{j+1} \int_0^\infty \frac{t^j e^{-zt}}{1-e^{-t}} dt
$$
for $z\in\mathbb C$ with $\Re z>0$ are the polygamma functions, see \cite[6.4.1]{AbramowitzStegun:1964}.
The result of this section is:

\begin{theorem}\label{thmRM}
Let $(a_n)_{n\in\mathbb N}$ be a sequence in $\mathbb R$ such that
$1\ll a_n \ll \sqrt{\log n}$ holds. The sequence of random variables 
$\Bigl(\frac{\Re \log Z}{a_n \sqrt{\log{n}}}\Bigr)_{n\in\mathbb N}$ and
$\Bigl(\frac{\Im \log Z }{a_n \sqrt{\log{n}}}\Bigr)_{n\in\mathbb N}$ under the average over the $C \beta E$ of $n \times n$ matrices
satisfy a moderate deviation principle for $\beta=1,2$ and $4$ with speed $a_n^2$ and rate function
$I(x)=\frac{x^2}{2}$.
\end{theorem}

\begin{remark} \label{remarkchar}
Theorem~\ref{thmRM} for $\beta=2$ states the same moderate deviation principle as
in \cite[page 440, Theorem 3.5]{HughesKeatingOConnell:2001} for the same scaling
range $1\ll a_n \ll \sqrt{\log n}$ -- but the speed in Theorem~\ref{thmRM} here is
given more explicit:
The speed $b_n$ of moderate deviations in
\cite[Theorem 3.5]{HughesKeatingOConnell:2001} is given by
$b_n= - \frac{a_n^2 \sigma_n^2}{W_{-1}\bigl(-\frac{a_n \sigma_n}{n}\bigr)}$,
where $W_{-1}$ denotes the Lambert's $W$-Function.
The Lambert's $W$-Function solves the equation $W(x) e^{W(x)}=x$ and $W_{-1}$
denotes the real branch with $W_{-1}(x)\leq -1$. For negative $x$
tending to zero we get the following asymptotic behaviour:
$W_{-1}(x) = \log |x| + {\mathcal O}\bigl(\log\bigl| \log|x| \bigr|\bigr) $.
This implies that the limiting speed behaves like
$$
b_n= - \frac{a_n^2 \sigma_n^2}{W_{-1}\bigl(-\frac{a_n \sigma_n}{n}\bigr)}
\sim \frac{a_n^2 \sigma_n^2}{\log n - \log{a_n \sigma_n}}
\sim \frac{a_n^2 \sigma_n^2}{\log n}\sim a_n^2 = s_n\,.
$$
Additionally in \cite[Theorem 3.5]{HughesKeatingOConnell:2001} 
the asymptotic behaviour of $\frac{\Re \log(Z)}{a_n \sqrt{\log{n}}}$ for scaling ranges $a_n=\sqrt{\log n}$
and $\sqrt{\log n} \ll a_n \ll n / \sqrt{\log n}$ is considered.
The circular orthogonal and circular symplectic ensembles were not studied
in \cite{HughesKeatingOConnell:2001}.
\end{remark}

\begin{proof}[Proof of Theorem \ref{thmRM}]
In \cite[eq. (47)]{KeatingSnaith:2000} an integral
representation of the cumulants of $\Re \log(Z)$ for the case $\beta=2$ is derived and
an outline of the extension to $\beta=1$ and $4$ is given. Similarly we prove a bound of the
cumulants satisfying the condition \eqref{eqcumulants} for these three circular
ensembles. With \eqref{vorbereitung} the cumulant can be written as
\begin{eqnarray*}
\Gamma_j\bigl(\Re \log(Z)\bigr)
&=& \frac{2^{j-1}-1}{2^{j-1}} \sum_{k=0}^{n-1} \psi^{(j-1)}\bigl(1+k \frac{\beta}{2}\bigr)
\\
&=& \frac{2^{j-1}-1}{2^{j-1}} \sum_{k=0}^{n-1} (-1)^{j} \int_0^{\infty} \frac{t^{j-1} e^{-(1+k\frac{\beta}{2})t}}{1-e^{-t}} dt
\\
&=& \frac{2^{j-1}-1}{2^{j-1}} (-1)^{j} \int_0^{\infty} t^{j-1} \frac{e^{-t}}{1-e^{-t}} \frac{1-e^{-n \frac{\beta}{2} t}}{1-e^{-\frac{\beta}{2}t}} dt
\\
&=& \frac{2^{j-1}-1}{2^{j-1}} (-1)^{j} \int_0^{\infty} t^{j-1} e^{-t} \bigl(1-e^{-n \frac{\beta}{2} t}\bigr) \sum_{r=0}^{\infty} \sum_{s=0}^{\infty} e^{-(s+r \frac{\beta}{2})t} dt
\end{eqnarray*}
using properties of geometric series for the last two equalities.
Thus we have
$$
\Gamma_j\bigl(\Re \log(Z)\bigr)
= \frac{2^{j-1}-1}{2^{j-1}} (-1)^{j} \sum_{r=0}^{\infty} \sum_{s=1}^{\infty} \int_0^{\infty} t^{j-1} e^{-(s+r \frac{\beta}{2})t} \bigl(1-e^{-n \frac{\beta}{2} t}\bigr) dt.
$$
To get a representation via the gamma function we integrate by substitution
needing a prefactor $(s+r \frac{\beta}{2})^{j-1}$ for $t^{j-1}$ and the
derivative $(s+r \frac{\beta}{2})$ of $(s+r \frac{\beta}{2})t$:
\begin{eqnarray}
\Gamma_j\bigl(\Re \log(Z)\bigr)
&=& \frac{2^{j-1}-1}{2^{j-1}} (-1)^{j} \Gamma(j) \left(
      \sum_{r=0}^{\infty} \sum_{s=1}^{\infty} \frac{1}{(s+r \frac{\beta}{2})^j} - 
      \sum_{r=n}^{\infty} \sum_{s=1}^{\infty} \frac{1}{(s+r \frac{\beta}{2})^j}
\right)
\nonumber\\
&\leq& \frac{2^{j-1}-1}{2^{j-1}} (-1)^{j} \Gamma(j)
      \sum_{r=0}^{\infty} \sum_{s=1}^{\infty} \frac{1}{(s+r \frac{\beta}{2})^j}\,.
\label{cumulantenzwischenstand}
\end{eqnarray}

In the case $\beta=1$ we can estimate the sum as follows:
For $r\in \mathbb N_0$ and $s\in \mathbb N$ the integer $k= 2 (s+\frac{r}{2})=2s+r$ can be
written in $k/2$ number of ways if $k$ is even, in no way if $k=1$, and in $\frac{k+1}{2}$
ways otherwise.
\begin{eqnarray*}
\sum_{r=0}^{\infty} \sum_{s=1}^{\infty} \frac{1}{(s+\frac{r}{2})^j}
&=&  2^j \left( \sum_{k=1}^{\infty} \frac{k}{(2k)^j}
     + \sum_{k=1}^{\infty} \frac{k}{(2k+1)^j}\right)
\\
&\leq& 2^{j-1} \left(
     \sum_{k=1}^{\infty} \frac{2k}{(2k)^j}
     + \sum_{k=1}^{\infty} \frac{2k-1}{(2k-1)^j}
     + \sum_{k=1}^{\infty} \frac{1}{(2k-1)^j}
\right)
\\
&=& 2^{j-1} \left(\zeta(j-2) + \bigl(1-\frac{1}{2^j}\bigr) \zeta(j-1) \right),
\end{eqnarray*}
applying the fact that
$
\sum_{k=1}^{\infty} (2k-1)^{-j}
= \sum_{k=1}^{\infty} k^{-j} - \sum_{k=1}^{\infty} (2k)^{-j} 
= \bigl(1-\frac{1}{2^j}\bigr) \zeta(j-1)
$.
Bounding the zeta function by $\frac{\pi^2}{6}$, this gives
$$
\sum_{r=0}^{\infty} \sum_{s=1}^{\infty} \frac{1}{(s+\frac{r}{2})^j}
\leq  2^{j-1} 2 \frac{\pi^2}{6} =  2^{j-1}  \frac{\pi^2}{3}.
$$

For $\beta=2$ we immediately get:
$$
\sum_{r=0}^{\infty} \sum_{s=1}^{\infty} \frac{1}{(s+r \frac{\beta}{2})^j}
= \sum_{k=1}^{\infty}  \frac{k}{k^j} = \zeta(j-1) \leq \frac{\pi^2}{6}.
$$

The case $\beta=4$ can be considered similarly, see \cite[p.84]{KeatingSnaith:2000}:
Counting the ways in which $k=s+2r$ this yields
$$
\sum_{r=0}^{n-1} \sum_{s=1}^{\infty} \frac{1}{(s+2r)^j}
\leq \frac{1}{2} \left(\zeta(j-1) + \bigl(1-\frac{1}{2^j}\bigr) \zeta(j) \right)
\leq \frac{\pi^2}{6}.
$$
Together with equation \eqref{cumulantenzwischenstand} we can conclude that
\begin{eqnarray*}
\left| \Gamma_j\Bigl(\frac{\Re \log(Z)}{\sigma_{n,\beta}}\Bigr) \right|
&=& \frac{\bigr| \Gamma_j \Re \log(Z) \bigr|}{\sigma_{n,\beta}^j}
\leq \frac{2^{j-1}-1}{2^{j-1}} \Gamma(j)
      \sum_{r=0}^{\infty} \sum_{s=1}^{\infty} \frac{1}{(s+r \frac{\beta}{2})^j}
\frac{1}{\sigma_{n,\beta}^j}
\\
&\leq& \Gamma(j) \frac{1}{\sigma_{n,\beta}^j}
\left\{\begin{array}{ll}
2^{j-1}  \frac{\pi^2}{3} & \text{for }\beta=1\\
\frac{\pi^2}{6} & \text{for }\beta=2,4.
\end{array}
\right.
\end{eqnarray*}
In order to read the parameters $\gamma$ and $\delta$ we apply that the variance of
$Z$ is bounded from below by
$\sigma_{n,\beta}^2 \geq \frac{\log 2}{\beta} \geq \frac{1}{2\beta}$.
Finally we have
\begin{equation}
\left| \Gamma_j\Bigl(\frac{\Re \log(Z)}{\sigma_{n,\beta}}\Bigr) \right|
\leq
(j!) \frac{1}{\sigma_{n,\beta}^{j-2}}
\left\{\begin{array}{ll}
2^j \frac{\pi^2}{3} & \text{for }\beta=1\\
4 \frac{\pi^2}{6} & \text{for }\beta=2\\
8 \frac{\pi^2}{6} & \text{for }\beta=4
\end{array}
\right\}
\leq
(j!) \frac{1}{\sigma_{n,\beta}^{j-2}}
\left\{\begin{array}{ll}
\bigl( \frac{8\pi^2}{3}\bigr)^{j-2} & \text{for }\beta=1\\
\bigl(\frac{2\pi^2}{3}\bigr)^{j-2} & \text{for }\beta=2\\
\bigl(\frac{4\pi^2}{3}\bigr)^{j-2} & \text{for }\beta=4
\end{array}
\right.
\end{equation}
for all $j\geq 3$, hence equation \eqref{eqcumulants} is satisfied for $\gamma=0$ and
$\Delta_n=  \frac{3 \sigma_{n,\beta}}{8\pi^2}$.
Theorem~\ref{thmcumulants} completes the prove for $\Re \log(Z)$.
Since the $j$-th cumulant of the distribution of the imaginary part of $\log Z$ can be bounded by
the $j$-th cumulant of the distribution of the real part of $\log Z$
for all $j\geq 3$, see \cite[eq. (62)]{KeatingSnaith:2000}, the MDP of
$\Im \log(Z)$ follows immediately.
\end{proof}

\begin{remark} 
Dyson observed that the induced eigenvalue distributions of the $C \beta E$ ensembles correspond to the Gibbs distribution
for the classical Coulomb gas on the circle at three different temperatures. Matrix models for general $\beta >0$
for Dysons's circular eigenvalue statistics are provided in \cite{KillipNenciu:2007}, using the theory of orthogonal polynomials
on the unit circle. They obtained a sparse matrix model which is five-diagonal. In this framework, there is no natural
underlying measure such as the Haar measure; the matrix ensembles are characterized by the laws of their elements. 
\end{remark}


\section{Moderate deviations for determinantal point processes}

The collection of eigenvalues of a random matrix can be viewed as a configuration of points 
(on ${\Bbb R}$ or on ${\Bbb C}$), that is a determinantal process. Central Limit Theorems
for occupation numbers were studied in the literature, see \cite{Zeitounibook}
and references therein. This section is devoted to the study of moderate deviation principles
for occupation numbers of determinantal point processes. We will see that it will be an application
of Theorem \ref{thmnotidentical}.

Let $\Lambda$ be a locally compact Polish space, equipped with a positive Radon measure $\mu$ on its
Borel $\sigma$-algebra. Let ${\mathcal M}_+(\Lambda)$ denote the set of positive $\sigma$-finite Radon measures on $\Lambda$.
A point process is a random, integer-valued $\chi \in {\mathcal M}_+(\Lambda)$, and it is simple if $P( \exists x \in \Lambda: \chi(\{x\}) >1)=0$.
A locally integrable function $\varrho : \Lambda^k \to [0, \infty)$ is called a joint intensity (correlation), if for
any mutually disjoint family of subsets $D_1, \ldots, D_k$ of $\Lambda$
$$
\E \bigl( \prod_{i=1}^k \chi(D_i) \bigr) = \int_{\prod_{i=1}^k D_i} \varrho_k(x_1, \ldots, x_k) d\mu(x_1) \cdots d \mu(x_k),
$$
where $\E$ denotes the expectation with respect to the law of the point configurations of $\chi$.
A simple point process $\chi$ is said to be a {\it determinantal point process} with kernel $K$ if its joint intensities $\varrho_k$
exist and are given by
\begin{equation} \label{DPP}
\varrho_k(x_1, \ldots, x_k) = \det_{i,j=1}^k \bigl( K(x_i, x_j) \bigr).
\end{equation}
An integral operator ${\mathcal K}: L^2(\mu) \to L^2(\mu)$ with kernel $K$ given by
$$
{\mathcal K}(f)(x) = \int K(x,y) f(y) \, d \mu(y), \quad f \in L^2(\mu)
$$
is {\it admissible} with admissible kernel $K$ if ${\mathcal K}$ is self-adjoint, nonnegative and locally trace-class
(for details see \cite[4.2.12]{Zeitounibook}). A standard result is, that an integral compact operator
${\mathcal K}$ with admissible kernel $K$ possesses the decomposition
$$
{\mathcal K} f(x) = \sum_{k=1}^n \lambda_k \phi_k(x) \langle \phi_k, f \rangle_{L^2(\mu)},
$$
where the functions $\phi_k$ are orthonormal in $L^2(\mu)$, $n$ is either finite or infinite, and $\lambda_k >0$
for all $k$, leading to
\begin{equation} \label{kernelrep}
K(x,y) = \sum_{k=1}^n \lambda_k \phi_k(x) \phi_k^*(y),
\end{equation}
an equality in $L^2(\mu \times \mu)$.
Moreover, an admissible integral operator ${\mathcal K}$ with kernel $K$ is called {\it good} with good kernel $K$ if the $\lambda_k$ in \eqref{kernelrep}
satisfy $\lambda_k \in (0,1]$. If the kernel $K$ of a determinantal point process is (locally) admissible, then it must in fact be good, see
\cite[4.2.21]{Zeitounibook}.

\begin{example} \label{GUEDDP}
If $(\lambda_1, \ldots, \lambda_n)$ be the eigenvalues of the GUE (Gaussian unitary ensemble) of dimension $n$ and denote
by $\chi_n$ the point process $\chi_n(D) = \sum_{i=1}^n 1_{\{ \lambda_i \in D\}}$. Then $\chi_n$ is a determinantal point process with admissible, good
kernel $K(x,y)= \sum_{k=0}^{n-1} \Psi_k(x) \Psi_k(y)$, where the functions $\Psi_k$ are the oscillator wave-functions, that is
$\Psi_k(x) := \frac{e^{-x^2/4} H_k(x)}{\sqrt{\sqrt{2 \pi} k!}}$, where $H_k(x):= (-1)^k e^{x^2/2} \frac{d^k}{dx^k} e^{-x^2/2}$ is the $k$-th
Hermite polynomial; see \cite[Def. 3.2.1, Ex. 4.2.15]{Zeitounibook}.
\end{example}

We will apply the following representation due to \cite[Theorem 7]{HoKPV06}: Suppose $\chi$ is a determinantal process with good kernel $K$ of the form
\eqref{kernelrep}, with $\sum_k \lambda_k < \infty$. Let $(I_k)_{k=1}^n$ be independent Bernoulli variables with $P(I_k=1) = \lambda_k$. Set
$$
K_I(x,y) = \sum_{k=1}^n I_k \, \phi_k(x) \phi_k^*(y),
$$
and let $\chi_I$ denote the determinantal point process with random kernel $K_I$. Then $\chi$ and $\chi_I$ have the same distribution.
Therefore, let $K$ be a good kernel and for $D \subset \Lambda$ we write $K_D(x,y)= 1_D(x) K(x,y) 1_D(y)$. Let $D$ be such that
$K_D$ is trace-class, with eigenvalues $\lambda_k$, $k \geq 1$. Then $\chi(D)$ has the same distribution as $\sum_k \xi_k$
where $\xi_k$ are independent Bernoulli random variables with $P(\xi_k=1)= \lambda_k$ and $P(\xi_k =0) = 1 - \lambda_k$.
Now we can state the main result of this section:

\begin{theorem} \label{mdpDDP}
Consider a sequence $(\chi_n)_n$ of determinantal point processes on $\Lambda$ with good kernels $K_n$. Let $D_n$
be a sequence of measurable subsets of $\Lambda$ such that $(K_n)_{D_n}$ is trace class. Assume that $(a_n)_n$ is a sequence of real numbers such that
$$
1 \ll a_n \ll \frac{\bigl( \sum_{k=1}^n \lambda_k^n(1- \lambda_k^n) \bigr)^{1/2}}{ \max_{1 \le i \le n} (\lambda_i^n(1-\lambda_i^n))^{1/2}}.
$$
Then $(Z_n)_n$ with
$$
Z_n := \frac{1}{a_n} \frac{\chi_n(D_n) - \E (\chi_n(D_n))}{\sqrt{\V (\chi_n(D_n))}}
$$
satisfies a moderate deviation principle with speed $a_n^2$ and rate function $I(x) = \frac{x^2}{2}$.
\end{theorem}

\begin{remark}
Obviously we have $\max_{1 \le i \le n} (\lambda_i^n(1-\lambda_i^n))^{1/2} \leq \frac 12$. To assure that $(a_n)_n$ is growing to infinity,
it is necessary that $\V (\chi_n(D_n))$ goes to infinity. Moreover, under the assumptions of Theorem \ref{mdpDDP},
$$
\V (\chi_n(D_n)) = \sum_k \lambda_k^n (1- \lambda_k^n) \leq \sum_k \lambda_k^n = \int K_n(x,x) d \mu_n(x),
$$
thus for a moderate deviation principle, it is necessary that $\lim_{n \to \infty} \int_{D_n} K_n(x,x) d\mu_n(x) = + \infty$.
\end{remark}

\begin{proof}[Proof of Theorem \ref{mdpDDP}]
We only have to check a moderate deviation principle for the rescaled partial sums of independent Bernoulli random variables $\xi_k$
with $P(\xi_k =1) = \lambda_k$. Therefore we apply Theorem \ref{thmnotidentical}. Take $X_k^n := \frac{\xi_k - \lambda_k^n}{\sqrt{\lambda_k^n(1- \lambda_k^n)}}$.
Then we obtain easily that condition \eqref{momentenbedingungen} is satisfied for $X_k^n$ with $\gamma=0$ and a constant $K_n=1$.
\end{proof}

\begin{example}[{\it Eigenvalues of the GUE/GOE}]
Let $D=[-a,b]$ with $a,b>0$ and $\alpha \in (-\frac 12, \frac 12)$, and $D_n := n^{\alpha} D$. Consider the determinantal point process of Example
\ref{GUEDDP}. Then $Z_n / a_n$ satisfies a moderate deviation principle; see \cite[4.2.27]{Zeitounibook}, where $\V( \chi_n(D_n)) \to \infty$ is proved applying
an upper bound with the help of the sine-kernel. Note, that the same conclusions hold when the GUE is replaced by the GOE (Gaussian orthogonal ensembles), see
\cite[4.2.29]{Zeitounibook}.
\end{example}

\begin{example}[{\it Sine-, Airy- and Bessel point processes}]
Recall the sine-kernel $K_{sine}(x,y)= \frac{1}{\pi} \frac{\sin(x-y)}{x-y}$ which arises as the limit of many interesting
point processes, for example as a scaling limit in the bulk of the spectrum in the GUE. 
With $\Lambda = {\Bbb R}$ and $\mu$ to be the Lebesgue measure, the corresponding operator is locally admissible
and determines a determinantal point process on ${\Bbb R}$. The operator is not of trace class but locally of trace class.
For $D_n = [-n,n]$, consider $K_n = 1_{D_n} K_{sine}$. The Central Limit Theorem for the rescaled $\chi_n(D_n)$ was proved by Costin and Lebowitz
in 1995. They proved that  $\V (\chi_n(D_n))$ goes to infinity. Hence a moderate deviation principle for the appropriately rescaled 
sine kernel process follows. It was shown in \cite{Soshnikov:2000}, that the condition $\lim_{n \to \infty} \V (\chi_n(D_n))= +\infty$
is satisfied for the Airy kernel $K_{Airy}$ with $D_n=[-n,n]$, and for Bessel kernel $K_{Bessel}$ with $D_n=[-n,n]$. In these cases, the growth
of $\V(\chi_n(D_n))$ is logarithmic with respect to the mean number of points in $D_n$. For a proof that the Airy process has a locally admissible
kernel which determines a determinantal point process, see \cite[4.2.30]{Zeitounibook}.
The Airy kernel arises as a scaling limit at the edge of the spectrum in the GUE and at the soft right edge of the spectrum in the Laguerre ensemble, 
while the Bessel kernel arises as a scaling limit at the hard left edge in the Laguerre ensemble. We conclude a moderate deviation principle
for the corresponding kernel point processes. For details and more examples like families of kernels corresponding to random matrices
for the classical compact groups, see \cite{Soshnikov2:2000}.
\end{example}

\section{Proof of Theorem \ref{thmcumulants}}\label{proofsection}
The following lemma is an essential element of the proof of
Theorem \ref{thmcumulants}. Rudzkis, Saulis and Statulevi{\v{c}}ius showed in 1978, that
condition \eqref{eqcumulants} on the cumulants implies the
following large deviation probabilities:

\begin{lemma}\label{lemmaRSS}
Let $Z$ be a centered random variable with variance one and existing
absolute moments, which satisfies 
$$
\bigl| \Gamma_j \bigr| \leq (j!)^{1+\gamma} / \Delta^{j-2}
\quad\text{for all } j=3,4, \dots
$$
for fixed $\gamma\geq 0$ and $\Delta>0$.
Then
\begin{eqnarray*}
\frac{P(Z\geq x)}{1-\Phi(x)}&=&
\exp\bigl( L_{\gamma}(x)\bigr)
\Bigl(1+q_1 \psi(x) \frac{x+1}{\Delta_{\gamma}}\Bigr)
\\\text{and }\quad
\frac{P(Z\leq- x)}{\Phi(-x)}&=&
\exp\bigl( L_{\gamma}(-x)\bigr)
\Bigl(1+q_2 \psi(x) \frac{x+1}{\Delta_{\gamma}}\Bigr)
\end{eqnarray*}
hold in the interval $0\leq x < \Delta_{\gamma}$,
using the following notation:
\begin{eqnarray}
\Delta_{\gamma}
&=& \frac{1}{6}\left( \frac{\sqrt{2}}{6} \Delta\right)^{1/(1+2\gamma)}
\nonumber
\\
\psi(x)
&=& \frac{60 \left(1+ 10 \Delta_{\gamma}^2 \exp\bigl( -(1-x/\Delta_{\gamma}) \sqrt{\Delta_{\gamma}}\bigr)\right)}{1-x/\Delta_{\gamma}}\,,
\label{eqpsi}
\end{eqnarray}
$q_1,q_2$ are two constants in the interval $[-1,1]$
and $L_{\gamma}$ is a function
(defined in \cite[Lemma 2.3, eq. (2.8)]{SaulisStratulyavichus:1989}) satisfying
\begin{equation}\label{eqLgamma}
\bigl| L_{\gamma}(x)\bigr| \leq \frac{|x|^3}{3 \Delta_{\gamma}}
\text{ for all $x$ with } |x|\leq \Delta_{\gamma}\,.
\end{equation}
\end{lemma}

For the proof see \cite[Lemma 2.3]{SaulisStratulyavichus:1989}.

\begin{lemma}\label{lemmaERS}
In the situation of Lemma \ref{lemmaRSS} there exist two constants
$C_1(\gamma)$ and $C_2(\gamma)$, which depend only on $\gamma$ and
satisfy the following inequalities:
\begin{eqnarray*}
&&\left| \log \frac{P(Z\geq x)}{1-\Phi(x)} \right| \leq C_2(\gamma) \frac{1+x^3}{\Delta^{1/(1+2\gamma)}} \\
\text{and}&& \left| \log \frac{P(Z\leq -x)}{\Phi(-x)} \right| \leq C_2(\gamma) \frac{1+x^3}{\Delta^{1/(1+2\gamma)}}
\end{eqnarray*}
for all $0\leq x\leq C_1(\gamma) \Delta^{1/(1+2\gamma)}$.
\end{lemma}

\begin{proof}
In \cite{ERS:2009} these bounds were concluded from the previous Lemma \ref{lemmaRSS}.
The proof here is analogue to the proof of \cite[Corollary 3.1]{ERS:2009}.
In the situation of Lemma \ref{lemmaRSS} the function $\psi$
defined in \eqref{eqpsi} is bounded by
$\psi(x) \leq c_1+ c_2\Delta_{\gamma}^2 \exp\bigl(-c_3\sqrt{\Delta_{\gamma}}\bigr)$
for all $0\leq x\leq q \Delta_{\gamma}$ for any fixed constant $q\in[0,1)$
and some positive constants $c_1,c_2$ and $c_3$ depending on $q$ only.
The term $c_1+ c_2\Delta_{\gamma}^2 \exp\bigl(-c_3\sqrt{\Delta_{\gamma}}\bigr)$
can be bounded uniformly in
$\Delta_{\gamma}
=\frac{1}{6}\left( \frac{\sqrt{2}}{6} \Delta\right)^{1/(1+2\gamma)}$,
which combined with the estimation \eqref{eqLgamma} implies the existence
of universal positive constants $c_4,c_5$ and $c_6$, such that
$$
\exp\Bigl( \frac{-c_5 x^3}{\Delta^{1/(1+2\gamma)}}\Bigr)
\Bigl( 1-\frac{c_6(1+x)}{\Delta^{1/(1+2\gamma)}}\Bigr)
\leq  \frac{P(Z\geq x)}{1-\Phi(x)} \leq
\exp\Bigl( \frac{c_5 x^3}{\Delta^{1/(1+2\gamma)}}\Bigr)
\Bigl( 1+\frac{c_6(1+x)}{\Delta^{1/(1+2\gamma)}}\Bigr)
$$
holds for all $0\leq x \leq c_4 \Delta^{1/(1+2\gamma)}$.
If $\Delta^{1/(1+2\gamma)}\leq 3 c_6$, we can choose $C_1(\gamma)$
and $C_2(\gamma)$ such that the first inequality in Lemma
\ref{lemmaERS} is satisfied.
In the case $\Delta^{1/(1+2\gamma)}> 3 c_6$ we have for all
$0\leq x \leq \frac{\Delta^{1/(1+2\gamma)}}{3 c_6}$
$$
\frac{c_6(1+x)}{\Delta^{1/(1+2\gamma)}} \leq
\frac{c_6}{\Delta^{1/(1+2\gamma)}} +\frac{1}{3}
\leq \frac{2}{3}\,.
$$
If $\Delta^{1/(1+2\gamma)}> 3 c_6$ and
$0\leq x \leq \frac{\Delta^{1/(1+2\gamma)}}{3 c_6}$ hold,
we can bound
\begin{equation*}
\left| \log \frac{P(Z\leq -x)}{\Phi(-x)} \right| \leq \frac{c_5 x^3}{\Delta^{1/(1+2\gamma)}}
+ \max\left\{ 
\Bigl|\log\Bigl( 1-\frac{c_6(1+x)}{\Delta^{1/(1+2\gamma)}}\Bigr)\Bigr|;
\Bigl|\log\Bigl( 1+\frac{c_6(1+x)}{\Delta^{1/(1+2\gamma)}}\Bigr)\Bigr|
\right\}\,.
\end{equation*}
Due to the concavity of the logarithm the absolute value of the straight line
\begin{eqnarray*}
g(x)=\frac{3 \log 3}{2} x - \frac{3 \log 3}{2}
&&\hspace{6.5cm}
\end{eqnarray*}
is bigger or equal
than the absolute value of $\log(x)$ for any
$\frac{1}{3}\leq x \leq \frac{5}{3}$. And we have
\begin{eqnarray*}
|\log(1-y)|\leq \frac{\log 3}{3/2} y = \frac{3 \log 3}{2} y
&&\hspace{6.5cm}
\end{eqnarray*}
and $|\log(1+y)|\leq y$
for any $0\leq y \leq \frac{2}{3}$.
Thus for $\Delta^{1/(1+2\gamma)}> 3 c_6$ and
$0\leq x \leq \frac{\Delta^{1/(1+2\gamma)}}{3 c_6}$ it follows that
\begin{equation*}
\left| \log \frac{P(Z\leq -x)}{\Phi(-x)} \right|
\leq \frac{c_5 x^3}{\Delta^{1/(1+2\gamma)}}
+ \frac{3 \log 3}{2} \frac{c_6(1+x)}{\Delta^{1/(1+2\gamma)}}
\leq \frac{c_5 x^3}{\Delta^{1/(1+2\gamma)}}
+ \frac{\log 3}{2} \frac{c_6(5+x^3)}{\Delta^{1/(1+2\gamma)}}\,,
\end{equation*}
applying $x^3-3x+2=(x-1)^2(x+2)\geq 0$ which is equivalent to
$3(1+x)\leq 5+x^3$.
Thus the first inequality in Lemma \ref{lemmaERS} is proved.
The second inequality in Lemma \ref{lemmaERS} can be proved similarly.
\end{proof}

\begin{proof}[Proof of Theorem \ref{thmcumulants}]
The idea of the proof is similarly to the proof of
\cite[Lemma 3.6]{ERS:2009} for the case of bounded geometric functionals.
It follows from Lemma \ref{lemmaERS} that in the situation of Theorem
\ref{thmcumulants} there exist two constants $C_1(\gamma)$ and $C_2(\gamma)$,
which satisfy the following inequalities:
\begin{eqnarray*}
&&\left| \log \frac{P(Z_n\geq y)}{1-\Phi(y)} \right| \leq C_2(\gamma) \frac{1+y^3}{\Delta_n^{1/(1+2\gamma)}} \\
\text{and}&& \left| \log \frac{P(Z_n\leq -y)}{\Phi(-y)} \right| \leq C_2(\gamma) \frac{1+y^3}{\Delta_n^{1/(1+2\gamma)}}
\end{eqnarray*}
for all $0\leq y\leq C_1(\gamma) \Delta_n^{1/(1+2\gamma)}$.
The logarithm can be represented as
\begin{eqnarray*}
\log \frac{P\left(\frac{1}{a_n} Z_n\geq x\right)}{1-\Phi(a_n x)}
&=& \log \frac{P\left(\frac{1}{a_n}Z_n\geq x\right)}{e^{\frac{(a_n x)^2}{2}} \bigl(1-\Phi(a_n x)\bigr)} e^{\frac{(a_n x)^2}{2}} \\
&=& \log P\left(\frac{1}{a_n}Z_n\geq x\right) +\frac{(a_n x)^2}{2}
-\log\Bigl({e^{\frac{(a_n x)^2}{2}} \bigl(1-\Phi(a_n x)\bigr)}\Bigr) \, .
\end{eqnarray*}
For the term at the left-hand side we can use the bounds
provided by Lemma \ref{lemmaERS} for
$y=a_n x$ and $0\leq x \leq C_1(\gamma) \frac{\Delta_n^{1/(1+2\gamma)}}{a_n}$.
Note that the bound for $x$ grows to infinity as $n$ does, thus it does not
imply any restriction.
Since, for all $y\geq 0$, we have
$$
\frac{1}{2+\sqrt{2\pi} y} \leq e^{\frac{y^2}{2}} \bigl(1-\Phi(y)\bigr) \leq \frac{1}{2}
$$
the monotonicity of the logarithm implies
\begin{eqnarray*}
\left|\log P\left(\frac{1}{a_n}Z_n\geq x\right) + \frac{(a_n x)^2}{2}\right|
&\leq& \left|\log\Bigl({e^{\frac{(a_n x)^2}{2}} \bigl(1-\Phi(a_n x)\bigr)}\Bigr)\right| + C_2(\gamma) \frac{1+(a_n x)^3}{\Delta_n^{1/(1+2\gamma)}}
\\
&\leq& \left|\log\left(\frac{1}{2+\sqrt{2\pi}a_n x}\right)\right| + C_2(\gamma) \frac{1+(a_n x)^3}{\Delta_n^{1/(1+2\gamma)}}
\\
&\leq& \log\bigl(2+\sqrt{2\pi}a_n x\bigr) + C_2(\gamma) \frac{1+(a_n x)^3}{\Delta_n^{1/(1+2\gamma)}}\, .
\end{eqnarray*}
And it follows that
\begin{eqnarray*}
&&\left|\frac{1}{a_n^2} \log P\left(\frac{1}{a_n}Z_n\geq x\right) + \frac{x^2}{2}\right|
\leq \frac{1}{a_n^2}  \log\bigl(2+\sqrt{2\pi}a_n x\bigr)
+ C_2(\gamma) \frac{1+(a_n x)^3}{a_n^2 \Delta_n^{1/(1+2\gamma)}}\\
&=&  \frac{1}{a_n^2}  \log(2+\sqrt{2\pi}a_n x)
+C_2(\gamma) \left(\frac{1}{a_n^2 \Delta_n^{1/(1+2\gamma)}}+\frac{a_n}{\Delta_n^{1/(1+2\gamma)}} x^3\right)
\stackrel{n\to\infty}{\longrightarrow} 0\, .
\end{eqnarray*}
Similarly we can prove
\begin{align*}
\left|\frac{1}{a_n^2} \log P\left(\frac{1}{a_n}Z_n\leq -x\right) + \frac{x^2}{2}\right|
&\leq \frac{1}{a_n^2}  \log\bigl(2+\sqrt{2\pi}a_n x\bigr)
+ C_2(\gamma) \frac{1+(a_n x)^3}{a_n^2 \Delta_n^{1/(1+2\gamma)}}\\
&\stackrel{n\to\infty}{\longrightarrow} 0\, .
\end{align*}
These bounds can be carried forward to a full moderate deviation principle
analogue to the proof of \cite[Theorem 1.2]{ERS:2009}.
\end{proof}

\newcommand{\SortNoop}[1]{}\def\cprime{$'$}
\providecommand{\bysame}{\leavevmode\hbox to3em{\hrulefill}\thinspace}
\providecommand{\MR}{\relax\ifhmode\unskip\space\fi MR }
\providecommand{\MRhref}[2]{%
  \href{http://www.ams.org/mathscinet-getitem?mr=#1}{#2}
}
\providecommand{\href}[2]{#2}

\end{document}